\documentclass[11pt]{article}

\usepackage{graphicx}%
\usepackage[utf8]{inputenc} 
\usepackage[T1]{fontenc}
\usepackage{multirow}%
\usepackage{amsmath,amssymb,amsfonts}%
\usepackage{mathtools}
\usepackage{amsthm}%
\usepackage{mathrsfs}%
\usepackage[normalem]{ulem}
\usepackage[title]{appendix}%
\usepackage{xcolor}%
\usepackage{indentfirst}
\usepackage{textcomp}%
\usepackage{manyfoot}%
\usepackage{booktabs}%
\usepackage{algorithm}%
\usepackage{algorithmicx}%
\usepackage{algpseudocode}%
\usepackage{listings}%
\usepackage{tikz}
\usepackage[colorlinks=true, allcolors=blue]{hyperref}
\usepackage{comment}
\newtheorem{theorem}{Theorem}

\newtheorem{definition}{Definition}
\newtheorem{proposition}{Proposition}

\newtheorem{remark}{Remark}%
\newtheorem{lemma}{Lemma}
\newtheorem{claim}{Claim}
\newcommand{\Z}{\mathbb{Z}}
\newcommand{\E}{\mathcal{E}}
\newcommand{\N}{\mathbb{N}}

\renewcommand{\P}{\mathbb{P}}

\DeclareGraphicsExtensions{.jpg,.pdf,.eps,.png}
\title{Constrained-degree percolation on the hypercubic lattice: uniqueness and some of its consequences}
\begin{document}
\sloppy
\author{
    Weberson S. Arcanjo\thanks{Departamento de Matemática Aplicada, Universidade Federal Fluminense, Niterói, Rio de Janeiro, Brazil.}\,\,\,\,\,\,\,\,\,\,\,\,\,
    Alan S. Pereira\thanks{Instituto de Matemática, Universidade Federal de Alagoas, Maceió, Alagoas, Brazil.}\,\,\,\,\,\,\,\,\,\,\,\,\,
    Diogo C. dos Santos$^\dagger$\,\,\,\,\,\,\,\,\,\,\,\,\,\\ \\
    Roger W. C. Silva\thanks{Departamento de Estatística, Universidade Federal de Minas Gerais, Belo Horizonte, Minas Gerais, Brazil.\newline}\,\,\,\,\,\,\,\,\,\,\,\,\,
    Marco Ticse$^\ddagger$\newline
}
\date{}
\maketitle
\vskip 15mm

\begin{abstract} We consider constrained-degree percolation on the hypercubic lattice. This is a continuous-time model defined by a sequence $(U_e)_{e\in\mathcal{E}^d}$ of i.i.d. uniform random variables and a positive integer $k$, referred to as the constraint. The model evolves as follows: each edge $e$ attempts to open at a random time $U_e$, independently of all other edges. It succeeds if, at time $U_e$, both of its end-vertices have degrees strictly smaller than $k$. It is known \cite{hartarsky2022weakly} that this model undergoes a phase transition when $d\geq3$ for most nontrivial values of $k$. In this work, we prove that, for any fixed constraint, the number of infinite clusters at any time $t\in[0,1)$ is almost surely either 0 or 1. We also show that the law of the process is differentiable with respect to time for local events, extending a result of \cite{SSS}. As a consequence of these two results, we prove that the percolation function is continuous in the supercritical regime $t\in(t_c,1)$, where $t_c$ denotes the percolation critical threshold. Finally, we show that the two-point connectivity function is bounded away from zero in the supercritical regime. 
\medskip


\end{abstract}

\section{Introduction}\label{sec1}

The study of random systems with local constraints dates back to the 1930s. One early example is the random sequential adsorption model introduced by Flory \cite{flory1939intramolecular}, motivated by problems in physical chemistry. A few decades later, the classical Bernoulli percolation model was proposed by Broadbent and Hammersley \cite{broadbent1957percolation} to describe the flow of a deterministic fluid through a random medium. Since then, percolation theory has given rise to a rich variety of models, with applications in fields as diverse as epidemiology \cite{ziff2021percolation}, composite materials \cite{eletskii2015electrical}, and forest fire dynamics \cite{perestrelo2022multi}.

The intersection of percolation theory and local constraints has attracted attention since at least the work of Gaunt, Guttmann, and Whittington \cite{DSGaunt_1979}, who introduced a percolation model with restricted valence—closely related to the model studied in this paper. In recent years, constrained random structures have gained renewed interest. For instance, Grimmett and Janson \cite{grimmett2010random} studied the Erdős–Rényi random graph $G_{n,p}$ conditioned on vertex degrees lying in a fixed subset $S$ of nonnegative integers. The 1-2 model, in which every vertex on the hexagonal lattice has degree 1 or 2, has been the subject of rigorous mathematical analysis \cite{grimmett20171}. Kinetically constrained models—particle systems introduced to study glassy and granular materials—have also received considerable attention \cite{10.1214/19-AAP1527}. Other examples of models with degree or interaction constraints appear in \cite{Dereudre,garet2018does,holroyd2021constrained,li2020constrained}.

\subsection{The model}

Let $\mathbb{L}^d=(\mathbb{Z}^d,\mathcal{E}^d)$ denote the hypercubic lattice with nearest neighbors. The \textbf{constrained-degree percolation} (CDP) we study in this work was introduced in \cite{Te}. It is a continuous-time percolation process with infinite-range dependencies, whose dynamics evolve as follows:  let $(U_e)_{e\in\mathcal{E}^d}$ be a sequence of $i.i.d.$ random variables with uniform distribution on the interval $[0, 1]$ and $k$ be a positive integer, $k\leq 2d$. At time $t = 0$, all edges are closed; an edge $e = \langle u, v\rangle$ opens at time $U_e$ if, for all $s<U_e$, it holds that  
\begin{equation*}
  \max\{|\{x\in\mathbb{Z}^d:\ \langle x, u\rangle \mbox{ is open at time }  s\}|, |\{x\in\mathbb{Z}^d:\ \langle x, v\rangle \mbox{ is open at time }  s\}| \} < k. \end{equation*}
Once an edge is open, it remains open.  We denote by $\mathbb{P}$ the corresponding probability measure on $[0,1]^{\mathcal{E}^d}$. Sometimes we refer to $(U_e)_{e\in\mathcal{E}^d}$ as a \textbf{collection of clocks}.

Given a collection of clocks $U=(U_e)_{e\in\mathcal{E}^d}$, let $\omega_t^{d,k}(U)\in\{0,1\}^{\mathcal{E}^d}$ denote the temporal percolation configuration of open and closed edges at time $t\in[0, 1]$. Specifically, $\omega_t^{d,k}(U,e)=1$ indicates that $e$ is open at time $t$, while $\omega_t^{d,k}(U,e)=0$ indicates that $e$ is closed at time $t$. In this context, we also refer to $e$ as \textbf{$t$-open} or \textbf{$t$-closed}, respectively. 

For simplicity, when there is no risk of confusion, we will omit the superscripts $d$ and $k$ and use the notation $\omega_t(U)$ and $\omega_t(U,e)$ instead.

Consider the CDP model on $\mathbb{L}^d$ with $k\leq 2d$. Let $\mathbb{G}_t$ be the subgraph of $\mathbb{L}^d$ induced by $\omega_t$. The connected components of $\mathbb{G}_t$ are called \textbf{open clusters}, and we say that two sites $x,y\in\Z^d$ are connected in $\omega_t$, denoted by $x\longleftrightarrow y$, if they are in the same open cluster of $\mathbb{G}_t$. If the cluster of a vertice $x$ is infinite, we say that $x$ is connected to infinity, denoted by $x\longleftrightarrow \infty$.

An important object in percolation theory is the probability that the origin belongs to an infinite cluster. With this in mind, define
\begin{equation*}
    \theta^{d,k}(t)=\mathbb{P}\left(0\longleftrightarrow \infty \mbox{ at $t$}\right).
\end{equation*}
Clearly, $\theta^{d,k}(t)$ is non-decreasing in $t$. Hence, we define the critical threshold $$t^{d,k}_c=\inf\{t\in[0,1]:\ \theta(t)>0 \}.$$


From now on, we omit the superscripts $d$ and $k$ from the notation and write simply $\theta(t)$ and $t_c$ instead of $\theta^{d,k}(t)$ and $t^k_c(\mathbb{L}^d)$. It is straightforward that the CDP is stochastically dominated by standard Bernoulli percolation for every fixed $t\in[0,1]$ and $k\leq 2d$, implying that the critical threshold $t_c$ is always greater than or equal to the critical parameter for Bernoulli percolation.

A nice feature of the CDP process is that translations of the lattice are ergodic with respect to $\mathbb{P}$. Since, for all $t\in[0,1]$, the event $\{\mbox{there exists an infinite cluster in } \omega_t\}$ is translation invariant, $\theta(t)$ is positive if, and only if,  
\begin{equation*}\label{eq: a outra função}
    \mathbb{P}\left(\mbox{there exists an infinite open cluster at $t$}\right)=1.
\end{equation*}

A non-trivial phase transition for the CDP model on $\mathbb{L}^2$ was first established in \cite{Te}, where it was shown that percolation occurs at $t=1$ for $k=3$. In \cite{LSSST}, the authors refined this result, proving that $1/2<t_c<1$ when $k=3$. Moreover, when $d>1$ and $k=2$, they show that $\theta(t)=0$, for all $t\in[0,1]$.

In \cite{hartarsky2022weakly}, the authors provide useful upper bounds for $t_c$ for several values of $d$ and $k$. More precisely, they prove that $t_c<1.7/d$ when $k>9$ and $d>k/2$; and in the following cases: $(d,k)=(4,7)$; $d\in\{5,6\}$ and $k\geq 8$; $d\in[7,16]$ and $k\geq9$. In the case $(d,k)=(3,5)$, they show that $t_c \leq 0.62$.

\subsection{Results }

The problem of uniqueness of the infinite open cluster in percolation models goes back to 1960. In that year, Harris \cite{harris1960lower} showed that Bernoulli percolation on the square lattice admits a unique infinite open cluster in the supercritical regime. Later, Newman and Schulman \cite{newman1981infinite} showed that, under the \textbf{finite energy condition}, the number of infinite open clusters in a percolation process on the $d$-dimensional hypercubic lattice must be 0, 1, or $\infty$, almost surely. The possibility of infinitely many infinite open clusters was  excluded in a work by Aizenman, Kesten, and Newman \cite{aizenman1987uniqueness}. Burton and Keane \cite{burton1989density} gave a beautiful argument for an alternative proof of this last result. 

The problem of uniqueness continues to attract the attention of many researchers. In \cite{10.1214/ECP.v19-3105}, the author shows that,  for any translation-invariant Gibbs measure of 1-2 model configurations
on the hexagonal lattice, the infinite homogeneous cluster is unique, almost surely. In \cite{curien2019uniqueness}, the authors consider the graph obtained by superposition of an independent pair of uniform infinite non-crossing perfect matchings of the set of integers. They prove that this graph contains at most one infinite path almost surely.  Other recent works on the topic include \cite{chebunin2024uniqueness, czajkowski2024non,severo2023uniqueness}. 

\begin{remark}\label{diff}
The CDP model presents several features that make its analysis particularly challenging. First, it does not satisfy the finite energy condition, which rules out many classical techniques based on local modifications. Second, the model lacks the FKG property, preventing the use of correlation inequalities that are central in standard percolation theory. Finally, it exhibits dependencies of arbitrary order, meaning that local events can be influenced by configurations far away. These aspects make the study of the CDP model both mathematically rich and technically demanding.
\end{remark}

Below, we present our results for the CDP model. Note that when $k = 2d$, the model reduces to the classical Bernoulli percolation process, for which the corresponding theorems are well known.

The main result of this work establishes the uniqueness of the infinite cluster for the CDP model on $\mathbb{L}^d$ with constraint $k \leq 2d$, for all $t \in [0,1)$. The first result in this direction was obtained in \cite{Te}, where uniqueness of the infinite cluster was proved for the CDP on $\mathbb{L}^d$, with $d \geq 2$, $k = 2d - 1$, and $t = 1$. Further progress was achieved in \cite{LSSST}, where the authors showed uniqueness of the infinite cluster for all $t$ such that $\theta(t) > 0$, in the case $d = 2$ and $k = 3$. We now state our main theorem.

\begin{theorem}\label{thm: trifurcação}
Consider the CDP model on $\mathbb{L}^d$ with constraint $k\leq 2d$. Let $\mathcal{N}_t$ denote the number of infinite clusters at time $t$. Then, for all $t\in[0,1)$, it holds that
\begin{equation*}
    \mathcal{N}_t \in\{0,1\}, \text{ almost surely.}
\end{equation*}
\end{theorem}

Our second result concerns the connectivity function of two points. Given $x,y\in \mathbb{Z}^d$, define
\begin{equation*}
\tau_t(x,y)=\P(x\longleftrightarrow y \mbox{ at time $t$}).
\end{equation*}
Theorem \ref{thm: trifurcação} allows us to prove the following result.

\begin{theorem}\label{conn1} Consider the CDP model on $\mathbb{L}^d$ with constraint $k\leq 2d$. If $\theta(t)>0$, then $$\displaystyle\inf_{x,y\in\mathbb{Z}^d}\tau_{t}(x,y)>0.$$
\end{theorem}


Clearly, $\theta(t)$ is continuous for $t<t_c$, since $\theta(t)=0$ in this case. It turns out that the continuity of $\theta(t)$ for $t>t_c$ is a consequence of the uniqueness of the open cluster and of the continuity of $t \mapsto \mathbb{P}(\omega_t^{-1}(\cdot))$ for local events. In Section \ref{analy}, we show that, given any local event $A$, the function $t \mapsto \mathbb{P}(\omega_t^{-1}(A))$ is differentiable in the time interval $[0,1]$, improving upon Theorem 3 of \cite{SSS}, which establishes only continuity. This will enable us to prove the following.

\begin{theorem}\label{cont} Consider the CDP model on $\mathbb{L}^d$ with constraint $k\leq 2d$. The percolation function $\theta(t)$ is a continuous function of $t$ on the interval $(t_c,1)$.
\end{theorem}


\subsection{Notation}

For $n\in\N$ and $v=(v_1,v_2,\ldots,v_d)\in\mathbb{Z}^d$, denote $[n]=\{1,\ldots,n\}$ 
 and $||v||=\sup\{|v_i|: i \in [d]\}$. For $A\subset \Z^d$,  we write the external edge boundary and inner vertex boundary of $A$ as    
$$
\partial^e A=\{\langle u, v \rangle \in \mathcal{E}^d \colon\ u \in A \mbox{ and } v\in A^c\},$$
and
$$\partial A=\{v\in A: \mbox{ there exits } u \in A^c \mbox{ such that } u \sim v\},$$
respectively,  where $u\sim v$ means $||u-v||=1$. We denote by
\begin{equation*}
    \mathcal{E}(A)=\{\langle u,v\rangle\in\mathcal{E}^d: u,v\in A\}
\end{equation*} the set of edges with both endpoints in $A$.
We write $\Lambda_n=\{v\in\Z^d: ||v|| \leq n\}$ for the box with side-length $2n$ and centre at the origin.

The restriction of a percolation configuration $\omega$ to a set $K\subset\mathcal{E}^d$ is denoted by $\omega|_K=\{\omega(e)\colon e\in K\}.$ Define the cylinder set
$[\omega]_{K}=\{\omega'\colon\omega'|_K=\omega|_K\}.$
We say that the event $A$ \textbf{lives on} a set $K$ if $\omega\in A$ and $\omega'\in [\omega]_{K}$ imply $\omega'\in A$.  Given $K\subset\mathcal{E}^d$, we write $\mathcal{H}_{K}=\{A\colon A\mbox{ lives on }K\}$. We say $A$ is a \textbf{local event} if $A$ lives on some finite set $K$.

The remainder of this paper is organized as follows. In Section \ref{sec: decreasing cluster}, we introduce the notion of a decreasing cluster. We also show that, when the configuration on the exterior edge boundary of a given box is fixed, the property of the box being connected to infinity, under suitable conditions, depends only on the clocks outside the box (Lemma \ref{teo:PMF}). In Section \ref{sec: proof of the main theorem}, we present the proof of Theorem \ref{thm: trifurcação} in two steps: first, we show that the number of infinite clusters is almost surely either $0$, $1$, or $\infty$; second, we rule out the case $\mathcal{N}_t = \infty$. In Section \ref{conn}, we prove Theorem \ref{conn1}. In Section \ref{analy}, we show that the function  $t \mapsto \mathbb{P}(\omega_t^{-1}(\cdot))$ is differentiable on $[0,1]$ for any local event, and we prove Theorem \ref{cont}. Finally, in Section \ref{proof_claims}, we prove some auxiliary results used to prove Theorem \ref{thm: trifurcação}.

\section{The decreasing cluster method}\label{sec: decreasing cluster}

As highlighted in Remark \ref{diff}, a primary challenge of this study is the presence of dependencies of all orders in the CDP model. To address this issue, we focus on analyzing the \textbf{decreasing cluster of an edge}, a concept defined in detail below.

Given a collection of clocks $U \in [0,1]^{\mathcal{E}^d}$ and two edges $e$ and $f$, we say that \textbf{$e$ is bigger than $f$} (with respect to $U$) if $U(e) > U(f)$.  A self-avoiding path $P$, consisting of edges $e_0, e_1,...,e_{n}$, is called \textbf{decreasing} if, for all $ i \in \{1,\dots, n\}$, the edge $e_{i-1}$ is bigger than $e_{i}$. 

For $\Lambda\subset\mathcal{E}^d$, $e\in\mathcal{E}(\Lambda)$, and $U\in [0,1]^{\mathcal{E}^d}$,  we define the \textbf{$\Lambda$-decreasing cluster} of $e$ on $U$ as the set $\mathcal{C}^{\Lambda}_{e} (U)$ of edges $f \in \Lambda$ such that there exists a decreasing path, entirely contained in $\Lambda$, starting from $e$ and ending in $f$. For convenience, we assume that $e$ belongs to its decreasing cluster. When $\Lambda = \mathcal{E}^d$, we omit $\Lambda$ from the notation and simply write $\mathcal{C}_{e}(U)$. Also, for $\Gamma \subset \E^d$, we write 
\begin{equation}\label{cluster2}
    \mathcal{C}_{\Gamma}(U)=\bigcup_{e \in \Gamma}\mathcal{C}_e(U). 
\end{equation}

We stress that the decreasing cluster depends only on the clock variables of the edges and not on their states.

Several authors have used an idea similar to that of the decreasing cluster. For instance, in \cite{AM} and \cite{SS}, this idea is employed to demonstrate that correlations between local events decay at least exponentially fast in the majority percolation and CDP models, respectively. More recently, \cite{coletti2024fluctuations} applied the decreasing cluster framework in the context of the parking process; see also \cite{ritchie2006construction}. Additionally, \cite{ráth2024locallimitrandomdegree} discusses its application in the analysis of random degree models.

Given $\Gamma \subset \E^d$, let $\mathcal{F}_{\Gamma}=\{U:|\mathcal{C}_{\Gamma}(U)|<\infty\}$. In what follows, we show that $\mathcal{C}_{\Gamma}$ is almost surely finite.

\begin{proposition}\label{prop:clusterfinito}
For every finite $\Gamma \subset \E^d$, it holds that $\mathbb{P}(\mathcal{F}_{\Gamma})=1$.
\end{proposition}

\begin{proof}
For $e \in \mathcal{E}^d$ and $n \in \N$, let $A_{n,e}$ be the event that there is a decreasing self-avoiding path of size $n$ starting at $e$. Note that $\{|\mathcal{C}_e|= \infty\} \subset A_{n,e}$, for all $n \in \mathbb{N}$. Hence,
\begin{equation*}
\mathbb{P}(|\mathcal{C}_e|= \infty) \leq \frac{2d(2d-1)^{n-1}}{n!},\mbox{ for all }n\in\N.
\end{equation*}
Since
\begin{equation*}
\mathcal{F}_{\Gamma}^c = \bigcup_{e \in \Gamma} \{|\mathcal{C}_e|= \infty\}.
\end{equation*}
the proof is complete.
\end{proof}

Next, we will present the main result of this section. Before that, we need to introduce further notation. For $E \subset \mathcal{E}^d$, let $$\Pi_{E^c}:[0,1]^{\mathcal{E}^d}\longrightarrow[0,1]^{E^c}$$ be the canonical projection of $[0,1]^{\mathcal{E}^d}$ onto $[0,1]^{E^c}$. That is, for each $U\in[0,1]^{\mathcal{E}^d}$, $\Pi_{E^c}(U)$ is the restriction of $U$ to the set of edges $E^c$. We say that $R$ is a \textbf{modification} of $U$ on $E$ if $R \in [0,1]^{E} \times {\Pi_{E^c}(U)}$, meaning $R$ agrees with $U$ on all edges in $E^c$.

Given $y\in\Z^d$, we denote by $\deg_t(y,U)$ the number of open edges incident to $y$ in the configuration $\omega_t(U)$. In symbols,
\begin{equation*}\label{def: random degree}
    \deg_t(y,U)=\#\{e\in\mathcal{E}^d\colon \omega_t(U,e)=1 \mbox{ and } e\sim y \}.
\end{equation*}
\begin{definition}
Given $U\in[0,1]^{\mathcal{E}^d}$, let $R$ be a modification of $U$ on $\E(\Lambda_n) \cup \partial^e \Lambda_n$. We say that $R$ \textbf{simplifies the boundary of} $\Lambda_n$ in $U$ at time  $t$ if, for every $e = \langle x, y \rangle \in \partial^e \Lambda_n$, with $x \in \Lambda_n$, it holds that
\begin{enumerate}
    \item[1)] if $\omega_t(U,e) =0$, then $R(e)>t$;
     \item[2)] if $\omega_t(U,e) =1$, then $R(e)= U(e)$ and, for all $s < R(e)$, we have $\deg_s(x,R) < k$.
\end{enumerate}
\end{definition}


It is worth noting the property of \textbf{state temporal consistency}, which ensures that the state of any edge $g$ at time $s:=U(g)$ remains unchanged for all $r>s$. 
Consequently, if $U(g)=R(g)$,  then $ \omega_s(U,g) = \omega_s(R,g)$ if, and only if,  $\omega_r(U,g) = \omega_r(R,g)$ for all $r\geq s$.

The following result establishes that if a modification of a collection of clocks within a box satisfies properties 1) and 2) of Definition 1, then the states of the edges outside the box remain identical for both collections of clocks.

\begin{lemma} \label{teo:PMF}
 Let $U \in [0,1]^{\E}$, and let $R$ be a modification of $U$ in $\mathcal{E}(\Lambda_n) \cup \partial^e\Lambda_n$ that simplifies the boundary of $\Lambda_n$. Then, for fixed $t\in[0,1)$, it holds that
\begin{equation}\label{Markov}
\omega_t(R,f)=\omega_t(U,f), \: \text{ for all } f \in \mathcal{E}(\Lambda_n)^c.
\end{equation} 
\end{lemma}

\begin{proof}
Fix $t \in [0,1)$ and let 
\[
A = \{ f \in \partial^e \Lambda_n : \omega_t(U,f) = 1\}, \qquad 
G = \mathcal{E}(\Lambda_n) \cup (\partial^e \Lambda_n \setminus A).
\]
For each $f \in \partial^e \Lambda_n \setminus A$, we have $\omega_t(U,f)=0$. Since $R$ simplifies the boundary, it follows that $\omega_t(R,f)=0$ as well. If $f \in G_n^c$ and $U(f)>t$, then $R(f)=U(f)>t$, which implies $\omega_t(U,f)=\omega_t(R,f)=0$. Hence, it suffices to prove \eqref{Markov} for $f \in G_n^c$ such that $U(f)\le t$. 

Let us prove that $\omega_t(U,g)=\omega_t(R,g)$ for all $g$ in the decreasing cluster $\mathcal{C}_f^{\Gamma}(U)$ with $\Gamma=G^c$, which includes $f$. Since $\mathcal{C}_f^{\Gamma}(U)$ is \textit{a.s.} finite, we enumerate its edges as $f_1,\ldots,f_m$ so that 
$U(f_1)<\cdots<U(f_m)=U(f)\le t$.  
Let $F_i=\{f_1,\ldots,f_i\}$ and $s_i=U(f_i)=R(f_i)$.  
For each $i$, write $f_i=\langle x_i,y_i\rangle$ and let $v_i\in\{x_i,y_i\}$.

At time $s_i$, the state of $f_i$ depends only on the open edges adjacent to $v_i$ with smaller clock values. By temporal consistency, for both $T\in\{U,R\}$,
\[
\{h\in \partial^e\Lambda_n\setminus A : h\sim v_i,\, \omega_{s_i}(T,h)=1\}=\emptyset.
\]
Hence, the open edges incident to $v_i$ at time $s_i$ belong to $\mathcal{E}(\Lambda_n)\cup F_{i-1}$. 

If $v_i\in\partial\Lambda_n$, then $f_i\in A$, so that $\omega_t(U,f_i)=1$.  
Since $R$ simplifies the boundary, we have $\deg_t(v_i,R)<k$. If $v_i\notin\Lambda_n$, then no internal edge of $\Lambda_n$ is incident to $v_i$. Assuming inductively that $\omega_{s_j}(U,f_j)=\omega_{s_j}(R,f_j)$ for all $j<i$, we conclude that
\[
\{h\in F_{i-1}: h\sim v_i,\,\omega_{s_i}(U,h)=1\}
=\{h\in F_{i-1}: h\sim v_i,\,\omega_{s_i}(R,h)=1\},
\]
which implies $\omega_{s_i}(U,f_i)=\omega_{s_i}(R,f_i)$. By induction, this holds for every $i$, and hence for all edges in $\mathcal{C}_f^{\Gamma}(U)$, including $f$. Therefore, $\omega_t(U,f)=\omega_t(R,f)$ for all $f\in G^c$, proving the lemma.
\end{proof}

\section{Proof of Theorem \ref{thm: trifurcação}}\label{sec: proof of the main theorem}

In this section, we will show Theorem \ref{thm: trifurcação}. We shall divide the proof into two parts. First, we will show that the number of infinite clusters must be 0, 1, or $\infty$, almost surely. This will be addressed in Section \ref{sec:01infty}. Next, in Section \ref{sec:notInfty}, we will rule out the possibility of having infinitely many clusters, thereby completing the proof.

\subsection{The number of infinite clusters is either 0, 1, or infinity} \label{sec:01infty}

Our approach is reminiscent of the work of Newman and Schulman \cite{newman1981infinite}. Specifically, we assume, with positive probability, the existence of $m$ infinite open clusters, $2\leq m <\infty$. We then select a sufficiently large box that intersects all the existing clusters and modify the configuration within this box in such a way that, with positive probability, all clusters merge into $m-1$ infinite clusters, leading to a contradiction.

This strategy requires careful handling for two key reasons. First, the CDP process does not satisfy the finite-energy property, requiring a more elaborate construction to replace the original argument in \cite{newman1981infinite}. Second, the model exhibits infinite-range dependencies, meaning that local configuration changes within a box can potentially influence the configuration outside the box.   

Let $\mathcal{N}_t$ denote the number of infinite open clusters in the temporal configuration $\omega_t$, and define the event $$I_{t,m}=\{\mathcal{N}_t=m\},\,\,\,\,m=1,\dots,\infty.$$

\begin{proposition}\label{prop: como}
   For all $t\in[0,1)$, and for every fixed constraint $k\leq 2d$, it holds that
\begin{equation*}
    \P(I_{t,m})=0
\end{equation*}  
for all $1<m<\infty$.
\end{proposition}

\begin{proof} Throughout the proof, we will rely on a series of claims, the proofs of which are provided in Section \ref{proof_claims}. 

Fix $t\in[0,1)$. Suppose that $\mathbb{P}(I_{t,m}) > 0$ for some $1< m < \infty$. Since $\mathbb{P}$ is translation-invariant and ergodic, and the event $I_{t,m}$ is also translation-invariant, it follows that $\P(I_{t,m})=1$.

Let $\mathcal{A}_{t,n} = \{ \Lambda_{n} \: \text{ intersects all the $m$ infinite open clusters at time $t$} \}$. 
Since $\mathcal{A}_{t,n} \uparrow I_{t,m}$ when $n$ goes to infinity, there exists $n$ large such that the event $\mathcal{A}_{t, n}$ occurs with positive probability.

Now, for each infinite cluster, choose a vertex in its intersection with $\partial \Lambda_n$. If there is a choice for the vertices then we pick them according to some predetermined ordering of all vertices.
More precisely, for each $i\in[m]$, let $v_i \in \partial \Lambda_n$ be such that $v_i$ is connected to infinity outside $\Lambda_{n-1}$ (i.e., by a path using no edges in $\E(\Lambda_{n-1})$), with $v_i \not \leftrightarrow v_j$ for $i\ne j \in [m]$, and write $V = \{v_1, \dots, v_m\}$. Our goal is to modify the states of the edges inside the box $\Lambda_n$ in such a way that we connect two vertices of $V$. To attain this, it will be convenient to take each $v_i$ in the set $\{x \in \partial \Lambda_n \colon x \sim  y, \text{ for some } y \in \Lambda_{n-1}  \}$, so we can connect $v_1$ to $v_2$ using only edges inside the box.

Observe that if we fix $V$ as above, then an open path connecting $v_i$ to infinity outside $\Lambda_{n-1}$ may contain edges in $\E(\partial \Lambda_n)$. In this case, for each $v_i$, we choose a vertex $w_i \in \partial \Lambda_n$ such that $w_i$ is connected to infinity outside $\Lambda_{n}$, and set $W = \{w_1, \dots, w_n\}$. Let $\gamma_i$ denote the path with vertices in $ \Lambda_{n-1}^c$ connecting $v_i$ to $w_i$, and write $\Gamma = \{\gamma_1, \dots, \gamma_m\}$. Figure \ref{fig:Boxn} illustrates this construction; for simplicity, the components of $v_3, \dots, v_m$ are omitted.

Next, we modify the configuration in two steps. In the first step, we close the edges in $\E(\Lambda_n)$, except those in $\gamma_i$ which are kept open. The states of edges with at least one endpoint not in $\Lambda_n$ are kept unchanged. 
In the second step, we construct a path $\alpha$ entirely contained in $\Lambda_n$ that connects the components containing the vertices $v_1$ and $v_2$.

\begin{figure}[!h]
\begin{center}
\begin{tikzpicture}[scale=0.5]

\draw[-, thick]  (0,0) -- (10,0);
\draw[-, thick]  (10,0) -- (10,1);
\draw[-, thick]  (10,3) -- (10,7);
\draw[-, thick]  (10,10) -- (0,10);
\draw[-, thick]  (0,10) -- (0,3);
\draw[-, thick]  (0,1.5) -- (0,0);

\draw[-, thick]  (1,1) -- (9,1);
\draw[-, thick]  (9,1) -- (9,9);
\draw[-, thick]  (9,9) -- (1,9);
\draw[-, thick]  (1,9) -- (1,1);

\filldraw (10,10) circle (2.5pt);
\filldraw (0,1.5) circle (2.5pt);
\filldraw (9,1) circle (2.5pt);
\filldraw (0,6) circle (2.5pt);
\filldraw (1,6) circle (2.5pt);
\filldraw (10,1) circle (2.5pt);
\node[right] at (10,10) {$w_1$};
\node[left] at (0,1.5) {$w_2$};
\node[right] at (10,1) {$v_1$};
\node[left] at (9,1.4) {$u_1$};
\node[left] at (-0.05,6.25) {$v_2$};
\node[right] at (0.9,6.3) {$u_2$};
\node[above] at (0.5,10) {\large{$\Lambda_n$}};
\node[above] at (1.5,9) {\small{$\Lambda_{n-1}$}};
\node[right] at (13,5) {$\gamma_1$};
\node[left] at (-3,4) {$\gamma_2$};
\node[right] at (5,3) {$\alpha$};


\draw [densely dotted] [line width= 0.5 mm] (10,10) to[out=80, in=250] (15,15);
\draw [densely dotted] [line width= 0.5 mm] (0,1.5) to[out=240, in=15] (-8,-1);
\draw [loosely dashed] [line width= 0.5 mm] (0,1.5)-- (0,3);
\draw [loosely dashed] [line width= 0.5 mm] (10,10)-- (10,7);
\draw [loosely dashed] [line width= 0.5 mm] (10,3)-- (10,1);

\draw [loosely dashed] [line width= 0.5 mm](0,3) to [out=145, in=270] (-3,4);
\draw [loosely dashed] [line width= 0.5 mm] (10,3) to [out=20, in=270] (13,5);
\draw [loosely dashed] [line width= 0.5 mm] (13,5) to [out=90, in=280] (10,7);
\draw [loosely dashed] [line width= 0.5 mm] (-3,4) to [out=90, in=210] (0,6);
\draw [densely dotted] [line width= 0.5 mm] (1,6)--(0,6);
\draw [densely dotted] [line width= 0.5 mm](10,1)-- (9,1);

\draw [solid] [line width= 0.5 mm] (1,6) to [out=280, in=100] (9,1);

\end{tikzpicture}
\caption{\small{Boxes $\Lambda_n$ and $\Lambda_{n-1}$ with the sites $w_i, v_i$, for $i\in \{1,2\}$, highlighted. The paths $\gamma_1$ and $\gamma_2$ connect $v_1$ to $w_1$ and $v_2$ to $w_2$, respectively. We build a path $\alpha$ by modifying the states inside of $\Lambda_n$.}}\label{fig:Boxn} 
\end{center}
\end{figure}

Let us formalize the discussion above. 
Fix $V = \{v_1, \dots, v_m\} \subset \{x \in \partial \Lambda_n \colon x \sim  y, \text{ for some } y \in \Lambda_{n-1}  \}$, $W = \{w_1, \dots, w_m\} \subset \partial \Lambda_n$,  and $\Gamma = \{ \gamma_1, \dots, \gamma_m\}$, where $\gamma_i$ is a finite self-avoiding path connecting $v_i$ to $w_i$ outside $\Lambda_{n-1}$, $i\in[m]$. Let $A\subset\partial^e\Lambda_n$ and define $\mathcal{D}_{t,n}^{\Gamma,A} $ as the event where

\begin{enumerate}
\item[($a$)]  $v_i \longleftrightarrow \infty$ outside $\Lambda_{n-1}$ for every $i\in[m]$;

\item[($b$)] $v_i$ and $v_j$ are {\bf not} connected for $i\neq j$;

\item[($c$)] $\gamma_i$ is open for every $i\in[m]$;

\item[($d$)] $w_i \longleftrightarrow \infty$ outside $\Lambda_{n}$, for every $i\in[m]$;

\item[($e$)] $e$ is $t$-open for all $e\in A$;

\item[($f$)] $f$ is $t$-closed for all $f \in \partial^e\Lambda_n\cap A^c$.

\end{enumerate}

\begin{claim} \label{claim:DGamma}
There exist sets $V$, $W$, $\Gamma$, and $A$, as above, such that 
\begin{equation} \label{claim:C1}
\mathbb{P} \left( \mathcal{D}_{t,n}^{\Gamma,A} \right) >0,
\end{equation}
for all $n$ sufficiently large.
\end{claim}

Assuming Claim 1, we proceed with a convenient modification of the configurations in $\mathcal{D}_{t,n}^{\Gamma,A}$, which will be carried out in two steps.\\


\textit{Step 1.}
Let $U \in  \mathcal{D}_{t,n}^{\Gamma,A}$. Suppose there exists a configuration $U'$, obtained from $U$ by modifying the clocks on $\mathcal{E}(\Lambda_n) \cup \partial^e \Lambda_n$, such that, in $\omega_t(U')$, the following conditions hold:

\begin{enumerate}    
    \item[A1.] all edges of $\gamma_1$ and $\gamma_2$ are open;

    \item[A2.] all the remaining edges in $\E(\Lambda_n)$ are closed;
    
    \item[A3.] the states of all edges in $[\E(\Lambda_n)]^c$ are unchanged and coincide with their states in \( \omega_t(U) \);

    \item[A4.] $\mathcal{N}_t=m$.

\end{enumerate}

Observe that, in $\omega_t(U')$, the configuration outside $\Lambda_n$ remains identical to that of $\omega_t(U)$, and the paths connecting $v_1$ to $w_1$ and $v_2$ to $w_2$ stay open. Consequently, $v_1$ and $v_2$ continue to belong to infinite components. Although closing edges inside the box could, in principle, increase the number of infinite components, we disregard this possibility since $\mathcal{N}_t=m$ holds almost surely.

Let $\mathcal{U}'$ denote the set of clocks $U'$ for which there exists $U \in \mathcal{D}_{t,n}^{\Gamma,A}$ such that conditions A1, A2, and A3 occur.
We aim to show that $\mathbb{P}(\mathcal{U}') >0$. 
As before, let
\begin{equation*}
  G_n=\mathcal{E}(\Lambda_n)\cup (\partial^e\Lambda_n-A),  
\end{equation*}
and 
\begin{equation} \label{eq:abreG}
\mathcal{E}_{t,n}^{\Gamma,A}=\left\{T\in[0,1]^{G_n}\colon \bigcap_{\substack{e\in G_n\cap(\mathcal{E}(\gamma_1) \cup \mathcal{E}(\gamma_2))\\f\in G_n\cap(\mathcal{E}(\gamma_1) \cup \mathcal{E}(\gamma_2))^c}}\{T(e)\leq t~;~ T(f)>t\}  \right\}.
\end{equation}

With a slight abuse of notation, identifying $\mathbb{P}$ with its marginals, and using \eqref{claim:C1}, we obtain 
\begin{align*}
\mathbb{P}\left( \mathcal{E}_{t,n}^{\Gamma,A} \times \Pi_{G_n^c} \left(\mathcal{D}_{t,n}^{\Gamma,A}\right) \right)=\mathbb{P}\left( \mathcal{E}_{t,n}^{\Gamma,A}\right)\cdot \mathbb{P}\left(\Pi_{G_n^c}\left(\mathcal{D}_{t,n}^{\Gamma,A}\right) \right)>0,
\end{align*}
which we use as a replacement for the lack of the finite energy property (see also \eqref{eq: positive prob}).






 The next claim shows that $ \mathcal{E}_{t,n}^{\Gamma,A} \times \Pi_{G_n^c}\left(\mathcal{D}_{t,n}^{\Gamma,A} \right) \subset \mathcal{U}'$.

\begin{claim}\label{lem:modif1} 
For $U \in \mathcal{D}_{t,n}^{\Gamma,A}$, let $R \in \mathcal{E}_{t,n}^{\Gamma,A} \times \Pi_{G_n^c}\left(\mathcal{D}_{t,n}^{\Gamma,A} \right)$ be a modification of $U$ on the set $G_n$. Then,

\begin{itemize}
    \item[a)] $R$ simplifies the boundary of $U$. In particular, if $f \in \mathcal{E}(\Lambda_n)^c$, then
\begin{equation}\label{eq: fora ta igual}
\omega_t(R,f)=\omega_t(U,f);
\end{equation}  

\item[b)] the paths $\gamma_1$ and $\gamma_2$ remain open in  $\omega_t(R)$.

\end{itemize}

\end{claim}


Writing 
\begin{equation*}
     \mathcal{U}'= \mathcal{E}_{t,n}^{\Gamma,A} \times \Pi_{G_n^c}\left(\mathcal{D}_{t,n}^{\Gamma,A} \right),
\end{equation*}
we describe the next step.

\textit{Step 2.} We now proceed to the final step of the proof, where we connect the infinite clusters of the vertices $v_1$ and $v_2$. For a given $U' \in \mathcal{U}' $, let $U''$ be a modification of $U'$ on $\mathcal{E}(\Lambda_{n-1})\cup\partial^e\Lambda_{n-1}$ such that, in $\omega_t(U'')$, the following conditions are satisfied:

\begin{enumerate}
    \item[B1.] $\omega_t(U',e)=\omega_t(U'',e)$ for all $e\in\mathcal{E}(\partial \Lambda_n \cup \Lambda_n^c)$;
    
    \item[B2.] $v_1$ and $v_2$ are connected in $\omega_t(U'')$.
\end{enumerate}

Let $\mathcal{V}$ denote the set of clocks $U''$ for which there exists $U' \in \mathcal{U'}$ such that conditions B1 and B2 hold. Note that, if $T \in \mathcal{V}$, then there are $m-1$ infinite clusters in $\omega_t(T)$. We will show that $\mathbb{P}(\mathcal{V})>0$.

Since $v_1$ and $v_2$ are points on the face of $\Lambda_n$, there exist $u_1, u_2 \in \Lambda_{n-1}$ such that $u_1 \sim v_1$ and $u_2 \sim v_2$. Because $u_1, u_2 \in \Lambda_{n-1}$, there exists a self-avoiding path $\alpha$ within $\Lambda_{n-1}$ that connects $u_1$ to $u_2$. This path $\alpha$ can be empty if $u_1 = u_2$; see Figure \ref{fig:Boxn}. Define $\gamma$ as the concatenation
\begin{equation} \label{eq: defgamma}
    \gamma = \langle v_1, u_1\rangle \alpha \langle v_2,u_2\rangle .
\end{equation}

Note that $\gamma$ is a self-avoiding path connecting $v_1$ to $v_2$, with all its edges in $\E(\Lambda_n)\setminus\E(\partial\Lambda_n)$.
For $H_n=\partial^e\Lambda_{n-1}\cup\mathcal{E}(\Lambda_{n-1})$, define
\begin{equation*} \label{eq: abreF}
\mathcal{E}_{t,n}^{\gamma} =\left\{T\in[0,1]^{H_n}\colon \bigcap_{\substack{e\in  H_n\cap\mathcal{E}(\gamma)\\f\in H_n-\mathcal{E}(\gamma)}}\{T_e\leq t~;~T_f>t\}   \right\}.
\end{equation*}
As before, it follows that
\begin{align}\label{eq: positive prob}
\mathbb{P}\left( \mathcal{E}_{t,n}^{\gamma} \times \Pi_{H_n^c} \left(\mathcal{U}' \right) \right)>0.
\end{align}

\begin{claim} \label{lem: final}
Let $U' \in \mathcal{U}' $ and let $R \in \mathcal{E}_{\gamma,n}^{t} \times \Pi_{H_n^c}\left(\mathcal{U}' \right)$ be a modification of $U'$ on the set $H_n$. Then,

\begin{itemize}
    \item[a)] $v_1$ and $v_2$ are connected in $\omega_t(R)$;
    \item[b)] for all $e \in \mathcal{E}(\partial\Lambda_n\cup\Lambda_n^c)$, it holds that
\begin{equation}\label{eq: coincidem...}
     \omega_t(R,e) = \omega_t(U',e).
\end{equation}
\end{itemize}

\end{claim}

Claim \ref{lem: final} establishes that if $R \in \mathcal{E}_{t,n}^{\gamma} \times \Pi_{H_n^c} (\mathcal{U}')$, then $\omega_t(R)$ satisfy properties B1 and B2 above for some collection of clocks $U'\in\mathcal{U}'$. That is,  $\mathcal{E}_{t,n}^{\gamma} \times \Pi_{H_n^c} (\mathcal{U}') \subset \mathcal{V}$. In particular, $\mathbb{P}(\mathcal{V})>0
$. 

Observe that if $T \in \mathcal{V}$, the configuration $\omega_t(T)$ will contain exactly $m-1$ infinite clusters. Thus, the number of infinite components is $m-1$ with positive probability, leading to a contradiction.

\end{proof}

\subsection{Impossibility of the existence of an infinite number of infinite clusters } \label{sec:notInfty}

It remains to rule out the possibility that $\mathcal{N}_t=\infty$. To do so, we use an argument inspired by the work of Burton and Keane \cite{burton1989density}, combined with the ideas developed in the previous section.

We say that $v\in\Z^d$ is a \textbf{encounter point} for a configuration $\omega\in\{0,1\}^{\mathcal{E}^d}$ if the following conditions hold
\begin{itemize}
    \item[$i$)] $v$ belongs to an infinite cluster $C$ in $\omega$;
    \item[$ii$)]  the set $C-\{v\}$ has no finite components and exactly three infinite
components.
\end{itemize}

Burton and Keane \cite{burton1989density} showed the following geometric result: for any $\omega\in\{0,1\}^{\mathcal{E}^d}$ and any rectangle $R\subset\Z^d$, the number of encounter points of $\omega$ in $R$ is less than the number of points on the boundary of $R$. Therefore, denoting by $X_n$ the random variable that counts the number of encounter points in $\Lambda_n$, we have, for all $n\in\N$, that 
\begin{equation}\label{eq: fato de BK}
    X_n \leq |\partial \Lambda_n|.
\end{equation}

Denoting by $p$ the probability that the origin is an encounter point, and using the invariance of $\mathbb{P}$ with respect to the translations of the lattice, we obtain, for all $n\in\N$, that
\begin{equation}\label{eq: trivial}
\mathbb{E}[X_n] = p|\Lambda_n|,    
\end{equation}
where $\mathbb{E}$ denotes expectation with respect to $\P$.

Assuming the existence of infinitely many infinite clusters, we first establish that $p>0$. Then, using \eqref{eq: fato de BK}, \eqref{eq: trivial}, and the amenability of $\mathbb{L}^d$, we arrive at a contradiction.

Suppose $\mathbb{P}(\mathcal{N}_t=\infty)=1$. In this case, the event
\begin{equation*}
     \tilde{\mathcal{A}}_{t,n}:=\{U\in[0,1]^{\mathcal{E}^d}:  \Lambda_n \mbox{ intersects at least 3 infinite clusters in } \omega_t(U)  \},
\end{equation*}
has positive probability for $n$ sufficiently large. 

For such values of $n$, we can choose sets $V$, $W$, and $\Gamma$, with $|V| = |W| = 3$, and proceed as in the previous section to conclude that $\mathbb{P}(\mathcal{D}_{t,n}^{\Gamma,A}) > 0$ (see Claim \ref{claim:DGamma}). 

\begin{figure}[h!]
\begin{center}
\begin{tikzpicture}[scale=0.5]

\draw[-, thick]  (0,0) -- (10,0);
\draw[-, thick]  (10,0) -- (10,10);
\draw[-, thick]  (10,10) -- (0,10);
\draw[-, thick]  (0,10) -- (0,0);

\node[below] at (10,0) {\large{$\Lambda_n$}};
\filldraw (2,10) circle (2.5pt);
\node[right] at (0.8,10.5) {$v_1$};
\filldraw (0,8) circle (2.5pt);
\node[above] at (-0.5,8) {$v_2$};
\filldraw (10,1) circle (2.5pt);
\node[right] at (10,1.5) {$v_3$};
\filldraw (2,8) circle (2.5pt);
\node[below] at (2,8) {$v$};
\node[below] at (5.6,5) {$\beta$};
\node[above] at (1.6,8) {$\alpha$};

\draw[dotted][-, thick]  (2,10) -- (2,8);
\draw[densely dotted][-, thick]  (0,8) -- (2,8);
\draw [densely dotted][line width= 0.5 mm] (2,10) to [out=100, in=250] (10,15);
\draw [densely dotted] [line width= 0.5 mm] (10,1) to [out=0, in=180] (18,5);
\draw [loosely dashed] [line width= 0.5 mm](2,8) to [out=330, in=180] (10,1);
\draw [densely dotted][line width= 0.5 mm] (0,8) to [out=160, in=270] (-7,15);
\end{tikzpicture}
\caption{\small{$\Lambda_n$ after the modifications that generate the trifurcation point $v$.}}\label{fig:trifurcation} 
\end{center}
\end{figure}

Furthermore, by suitably modifying the clock values of the edges inside $\Lambda_n$, while preserving the states of all edges outside $\Lambda_n$, we can merge three infinite clusters at a vertex of $\Lambda_{n-1}$. In fact, we perform two successive modifications within the box $\Lambda_n$ in order to create, with positive probability, a vertex $v$ inside the box that becomes a trifurcation.

The first modification, as in Claim \ref{lem:modif1}, closes all edges inside $\Lambda_n$ while keeping the paths $\gamma_1$, $\gamma_2$, and $\gamma_3$ open (see Equation \eqref{eq:abreG}). The second modification connects $v_1$ and $v_2$ by opening a path $\alpha$; we then choose a fixed vertex $v \in \alpha$ and connect $v_3$ to $v$ by an open path $\beta$, in the same spirit as the construction in Claim \ref{lem: final} (see Figure \ref{fig:trifurcation}). These two adjustments guarantee that the vertex $v$ becomes a trifurcation.

This completes the proof of Theorem \ref{thm: trifurcação}.

\begin{remark} 
The proof of Theorem \ref{thm: trifurcação} relies on fixing the states of the boundary edges of the box $ \Lambda_n$.
The strategy used to keep these edges closed after locally modifying the clock times
was to ensure that they could become open only after the reference time $t$.
However, this approach presents a limitation when $t=1$.
Without fixing the boundary, modifications of the clock times inside the box
may alter the states of edges outside it, due to the long-range dependencies intrinsic to the model.
\end{remark}

\section{Proof of Theorem \ref{conn1}}\label{conn}

In this section, we prove Theorem \ref{conn1}, showing that the two-point connectivity function is uniformly bounded away from zero (on the set of pairs of vertices). Here, we write $\mathbb{C}_t(x)$ for the open cluster of vertex $x$ at time $t$, i.e., the set of vertices connected to $x$ by a path of $t$-open edges.

\textit{Proof of Theorem \ref{conn1}:} Given $x\in\Z^d$ and $n\in\N$, consider the events
$$\mathcal{L}_{t,n}(x)=\{x\longleftrightarrow\partial \Lambda_n\mbox{ at $t$},|\mathbb{C}_t(x)|<\infty\},$$ $$\mathcal{B}_{t,n}(x)=\{x\longleftrightarrow\partial \Lambda_n\mbox{ at $t$}\},$$ and write $\mathcal{B}_t(x)=\{x\longleftrightarrow\infty\mbox{ at time $t$}\}.$ Since $\{\mathcal{B}_{t,n}(x)\}_{n\geq 1}$ decreases to $\mathcal{B}_t(x)$ we obtain $\P(\mathcal{L}_{t,n}(x))=\P(\mathcal{B}_{t,n}(x))-\P(\mathcal{B}_{t}(x))\rightarrow 0$, when $n$ goes to infinity. Choose $n_0$ large enough such that 
\begin{equation*}\label{bound_3}
\P(\mathcal{L}_{t,n_0}(x))\leq \theta^2/4.
\end{equation*} This clearly gives $\P(\mathcal{L}_{t,n_0}^c(x)\cap\mathcal{L}_{t,n_0}^c(y))\geq 1-\tfrac{1}{2}\theta^2$.

Note that $\mathcal{B}_t(x)=\mathcal{B}_{t,n}(x)\cap\mathcal{L}_{t,n}^c(x)$, for all $n\geq 1$. Hence, by uniqueness of the infinite open cluster, we have
\begin{align*}\tau_t(x,y)\geq \P(\mathcal{B}_t(x)\cap\mathcal{B}_t(y))&= \P((\mathcal{B}_{t,n_0}(x)\cap \mathcal{L}_{t,n_0}^c(x))\cap(\mathcal{B}_{t,n_0}(y)\cap\mathcal{L}_{t,n_0}^c(y)))\\
&\geq \P(\mathcal{B}_{t,n_0}(x)\cap \mathcal{B}_{t,n_0}(y))-\tfrac{1}{2}\theta^2.
\end{align*}
Since $\mathcal{B}_{t,n}(x)$ and $\mathcal{B}_{t,n_0}(y)$ are local events, Theorem 2 in \cite{SSS}, together with translation invariance, yields  
$$\tau_t(x,y)\geq \P(\mathcal{B}_{t,n_0}(x))\P(\mathcal{B}_{t,n_0}(y))-c_1n_0e^{-\psi d(x,y)}-\tfrac{1}{2}\theta^2
\geq \tfrac{1}{2}\theta^2-c_1n_0e^{-\psi d(x,y)},$$
for some positive constants $\psi$ and $c_1$. This gives $\tau_t(x,y)>0$ for all $x,y\in\mathbb{Z}^d$ with $d(x,y)$ large enough, say $d(x,y)\geq k_0$. Since $\tau_{t}(x,y)$ is strictly positive for all $x,y$ with $d(x,y)<k_0$, the result follows.
\qed

\section{Differentiability of $\mathbb{P}$ for local events}\label{analy}

In this section, we establish the differentiability of $t \mapsto \mathbb{P}(\omega_t^{-1}(\cdot))$ for local events. This result, combined with the uniqueness of the infinite open cluster, will allow us to prove the continuity of the function $\theta(t)$ for $t \in (t_c^k, 1)$. Moreover, differentiability of $\mathbb{P}$ plays a key technical role in settings such as dynamic percolation or models where randomness is progressively revealed over time. In particular, when analyzing the sharpness of phase transitions via randomized algorithms on graphs, the differentiability of $\mathbb{P}$ is crucial (see \cite{DT,DRT}).

\begin{proposition}\label{diff2} Let $A\subset\{0,1\}^{\E}$ be a local event. Then, the function $ t \mapsto \P_t(A)=\P(\omega_t^{-1}(A))$ is differentiable on $(0,1)$.
\end{proposition}

\begin{proof} Assume that $A$ lives in $\Gamma\subset \E^d$. For $ t \in [0,1] $ and $ e \in \Gamma $, recall the definition in \eqref{cluster2} and write 

 \begin{equation}
 \mathcal{C}_{t,e} = \left \{
 \begin{array}{cc}
 \mathcal{C}_e, & \mbox{if } U_e \leq t, \\
 \{e\}, & \mbox{if } U_e > t. \\
 \end{array}
 \right.
 \end{equation}
Also, let

\[  \mathcal{C}_{t} = \bigcup_{ e \in \Gamma } \mathcal{C}_{t,e}.  \]

Denote by $\delta(u,v)$ the graph distance between the vertices $u,v\in\Z^d$. Given $\Gamma,\tilde{\Gamma}\subset\Z^d$, let $\delta(\Gamma,\tilde{\Gamma})=\inf\{\delta(x,y): x\in\Gamma, y\in\tilde{\Gamma}\}$, and write $B_r(\Gamma)=\{x\in\Z^d: \delta(x,\Gamma)\leq r\}$ for the ball of radius $r$ centered at $\Gamma$.

Let $r\in \N$, and consider the clock event

\begin{equation*}\Xi_{t,r }=\left\{ U \colon \mathcal{C}_{ t} \subset B_{r} (\Gamma) \right\}.
\end{equation*}
A standard counting argument (see \cite{AM} and \cite{SS}) yields the estimate
\begin{equation}\label{estimate}
\mathbb{P} \left( \Xi_{t,r}^c \right)  \leq  | \Gamma| \sum_{m = r}  ^{\infty} \frac{{(2dt)}^m}{m!} \leq  C(\Gamma,d) \frac{(2dt)^r}{r!},
\end{equation}
with $C=C(\Gamma,d)=| \Gamma| e^{2d}$.

From Proposition \ref{prop:clusterfinito}, we have
\begin{equation*}
    \bigcup_{r=0}^{\infty} \Xi_{t,r}= [0,1]^{\mathbb{E}^d}\setminus \mathcal{N},
\end{equation*}
with $\P(\mathcal N)=0$. Write $\Theta_{t,r} = \Xi_{t,r}-\Xi_{t,r-1}$, for $r \geq 1$, and $\Theta_{t,0} = \Xi_{t,0}$. Since $\Xi_{t,r-1} \subset \Xi_{t,r}$ for all $r\geq 1$, we obtain
\begin{equation*}
     \bigcup_{r=0}^{\infty} \Theta_{t,r}= [0,1]^{\mathbb{E}^d}\setminus\mathcal{N},
\end{equation*}
where the last union is disjoint. Hence,
\begin{align}\label{prob_A}
    \mathbb{P}_t(A)\coloneq\mathbb{P}(\omega_t^{-1}(A))&=\mathbb{P}\left(  \bigcup_{r \geq 0} \left( \omega_t^{-1}(A) \cap \Theta_{t,r}   \right) \right)\nonumber\\
                   &=\sum_{r \geq 0} \mathbb{P}  \left( \omega_t^{-1}(A)  \cap  \Theta_{t,r}  \right).
\end{align}

A slight modification of Lemma 2 in \cite{SSS} implies that the clock event $\omega_t^{-1}(A) \cap \Xi_{t,r}$ depends only on the random variables $U_e$ for $e \in \mathcal{E}(B_{r+1}(\Gamma))$. Hence, for $U, \tilde{U} \in [0,1]^{\mathcal{E}^d}$ satisfying $U_e = \tilde{U}_e$ for all $e$ in this set, $U \in \Xi_{t,r}$ implies $\tilde{U} \in \Xi_{t,r}$. The same holds for $\Theta_{t,r}$ and $\omega_t^{-1}(A) \cap \Theta_{t,r}$.

Let $m=m(r)=|\mathcal{E}(B_{r+1}(\Gamma))|$ and write $I_m=\{1,2,\ldots,m\}$. Consider the set 
\begin{equation*}
    \mathcal{P}_m=\{ \sigma: I_m \longrightarrow \mathcal{E}(B_{r+1}(\Gamma))\colon \sigma \mbox{ is one-to-one}\}.
\end{equation*}
Each element of $\mathcal{P}_m$ can be seen as an ordering of the edges of $\mathcal{E}(B_{r+1}(\Gamma))$.  For $\sigma \in \mathcal{P}_m$ and $\ell \in I_m^*=I_m \cup \{0\}$, define
\begin{equation*}
    \mathbb{E}_{t,r}^{\sigma,\ell}=\{U \in [0,1]^{\mathcal{E}(B_{r+1}(\Gamma))}\colon   U_{\sigma(1)} < \ldots <U_{\sigma(\ell)}<t<U_{\sigma(\ell+1)}<\ldots<U_{\sigma(m)}\}.
\end{equation*}
A direct calculation shows
\begin{equation*}
 \mathbb{P} \left(\mathbb{E}_{t,r}^{\sigma,\ell}\right) = \frac{t^\ell}{\ell!} \cdot \frac{(1-t)^{m-\ell}}{(m-\ell)!}.   
\end{equation*}
Since $\omega_t^{-1}(A) \cap \Theta_{t,r}$ depends only on the clocks in $\mathcal{E}(B_{r+1}(\Gamma))$, we obtain
\begin{equation*}
\omega_t^{-1}(A) \cap \Theta_{t,r} = \bigcup_{\sigma \in \mathcal{P}_m,\ \ell \in I_m^*} \left( \omega_t^{-1}(A) \cap \Theta_{t,r} \cap   \mathbb{E}_{t,r}^{\sigma,\ell} \right).
\end{equation*}

A key point is that if $\omega_t^{-1}(A)\cap\Theta_{t,r}\cap   \mathbb{E}_{t,r}^{\sigma,\ell}\neq\emptyset$, then $\mathbb{E}_{t,r}^{\sigma,\ell} \subset \omega_t^{-1}(A) \cap \Theta_{t,r}$. This follows from the fact that the occurrence of the local event $A$ depends only on the order in which clocks ring before and after time
$t$, rather than on their actual values. With this in mind, for $r\geq 0$, define
$$S_{t,r}=\{ (\sigma,\ell) \in \mathcal{P}_m \times I_m^{*}:  \mathbb{E}_{t,r}^{\sigma, \ell} \cap \omega_t^{-1}(A) \cap \Theta_{t,r} \neq \emptyset \},$$ and note that
\begin{equation*}
\omega_t^{-1}(A) \cap \Theta_{t,r} = \bigcup_{(\sigma,\ell) \in S_{t,r}}    \mathbb{E}_{t,r}^{\sigma,\ell}.
\end{equation*}
Fix $\ell$ and define $S_{t,r}^\ell := \{\sigma \in \mathcal{P}_m : (\sigma,\ell) \in S_{t,r}\}$. Then:
\begin{align*}
\mathbb{P}_t(A) &= \sum_{r = 0}^\infty \sum_{(\sigma,\ell) \in S_{t,r}} \mathbb{P}(\mathbb{E}_{t,r}^{\sigma,\ell}) \\
&= \sum_{r = 0}^\infty \sum_{\ell = 0}^{m(r)} \sum_{\sigma \in S_{t,r}^\ell} \frac{t^\ell}{\ell!} \cdot \frac{(1 - t)^{m - \ell}}{(m - \ell)!}.
\end{align*}

Define
\begin{equation}\label{efe}
f_r(t) := \sum_{\ell = 0}^{m(r)} |S_{t,r}^\ell| \cdot \frac{t^\ell}{\ell!} \cdot \frac{(1 - t)^{m - \ell}}{(m - \ell)!}.
\end{equation}
From \eqref{estimate} and \eqref{prob_A}, we have
\[
f_r(t) = \mathbb{P}(\omega_t^{-1}(A) \cap \Theta_{t,r}) \leq \mathbb{P}(\Xi_{t,r-1}^c) \leq C \cdot \frac{(2dt)^{r - 1}}{(r - 1)!}.
\]

Assuming that $|S_{t,r}^\ell|$ does not depend on $t$, we obtain (using \eqref{efe})
$$
|f_r'(t)| \leq \left( \frac{m(r)}{t(1 - t)} \right)\cdot f_r(t),
$$
for all $t\in(0,1)$. Hence, for every $\alpha, \beta \in (0,1)$ with $\alpha < \beta$, and $t \in [\alpha, \beta]$, we have $\frac{1}{t} \leq \frac{1}{\alpha}$ and $\frac{1}{1 - t} \leq \frac{1}{1 - \beta}$, yielding to
$$
|f_r'(t)| \leq \left( \frac{m(r)}{\alpha(1 - \beta)} \right) f_r(t)\leq C' m(r)\cdot \frac{(2d)^{r-1}}{(r-1)!},
$$
for all $t\in(0,1)$, $r\geq 2$, and some constant $C'$.

Since $ m(r) $ is polynomial in $ r $, the series $ \sum_{r\geq 0} f^{'}_r(t) $ converges uniformly on the interval $ (0,1) $ by the Weierstrass M-test, and  $\P_t(A)$ is differentiable on $(0,1)$.

To complete the proof, it remains to show that $ S_{t,r}^\ell $ does not depend on $t$. Let $0 < t_1 < t_2 < 1$ and fix $r \geq 0$, $\ell \in \{0, \ldots, m\}$. We claim that $S_{t_1,r}^\ell = S_{t_2,r}^\ell$.

Let $\sigma \in S_{t_1,r}^\ell$ and choose $U \in \omega_{t_1}^{-1}(A) \cap \Theta_{t_1,r} \cap \mathbb{E}_{t_1,r}^{\sigma,\ell}$. Define $V$ by setting $V(e) = U(e) + (t_2 - t_1)$ for all $e$. Then:
\begin{itemize}
\item[($a$)] $V \in \mathbb{E}_{t_2,r}^{\sigma,\ell}$,
\item[($b$)] $\mathcal{C}_e(U) = \mathcal{C}_e(V)$ for all $e$, hence $V \in \Theta_{t_2,r}$,
\item[($c$)] since the configuration remains the same, $V \in \omega_{t_2}^{-1}(A)$.
\end{itemize}

Therefore, $V \in \omega_{t_2}^{-1}(A) \cap \Theta_{t_2,r} \cap \mathbb{E}_{t_2,r}^{\sigma,\ell}$, implying $\sigma \in S_{t_2,r}^\ell$. The reverse inclusion follows similarly by subtracting $(t_2 - t_1)$, completing the proof. 
\end{proof}

\begin{remark} One can show that $t \mapsto \mathbb{P}(\omega_t^{-1}(A))$ is infinitely differentiable on $(0,1)$.  This follows since the $k$-th derivative $f_r^{(k)}(t)$ of $f_r(t)$ is bounded by
$$
f_r^{(k)}(t) \leq \left( \frac{m(r)}{t(1 - t)} \right)^k\cdot f_r(t),
$$
for all $k \in \mathbb{N}$. Hence, for every $\alpha, \beta \in (0,1)$ with $\alpha < \beta$, and $t \in [\alpha, \beta]$, we have $\frac{1}{t} \leq \frac{1}{\alpha}$ and $\frac{1}{1 - t} \leq \frac{1}{1 - \beta}$, yielding to
$$
f_r^{(k)}(t) \leq \left( \frac{m(r)}{\alpha(1 - \beta)} \right)^k f_r(t).
$$
The infinite differentiability follows with a second application of the Weierstrass M-test.
\end{remark}

We conclude this section with the proof of Theorem \ref{cont}.
\bigskip

\textit{Proof of Theorem \ref{cont}:} Recall the definitions of the events $\mathcal{B}_t(0)$ and $\mathcal{B}_{t,n}(0)$ from Section \ref{conn}. It is clear that $$\theta(t)=\lim_{n\rightarrow \infty}\P(\mathcal{B}_{t,n}(0)).$$ By Proposition \ref{diff2}, the function $\mathbb{P}(\mathcal{B}_{t,n})$ is continuous in $t$, so $\theta(t)$ is the decreasing limit of continuous functions. Therefore, $\theta(t)$ is upper semi-continuous. Since $\theta(t)$ is monotonic non-decreasing, it follows that $\theta(t)$ is continuous from the right on $[0,1]$. 

Continuity from the left on the interval 
$(t_c,1)$ follows from the classical argument of van den Berg and Keane \cite{BK}, combined with the uniqueness of the infinite open cluster established in Theorem \ref{thm: trifurcação}.
\qed

\section{Proofs of claims}\label{proof_claims}

In this section, we prove the claims in Section \ref{sec:01infty}. We start with the proof of Claim \ref{claim:DGamma}.

\begin{proof}[Proof of Claim \ref{claim:DGamma}]

For $n\in\N$, let $V = \{ v_1, \ldots, v_m \} \subset \{x \in \partial \Lambda_n \colon x \sim y \text{ for some } y \in \Lambda_{n-1} \} $ be such that, with positive probability, at time \( t \), the following conditions hold:
\begin{itemize}
\item[($a$)] \( v_i \) is connected to infinity outside \( \Lambda_{n-1} \) for every \( i \in [m] \);
\item[($b$)] $v_i$ and $v_j$ are {\bf not} connected for $i\neq j$.
\end{itemize}

Denote by \( \mathcal{X}_{t,n}^{V} \) the event where ($a$) and ($b$) occur at time \(t\). Recall that $$\mathcal{A}_{t,n} = \{ \Lambda_{n} \text{ intersects all $m$ infinite clusters at time $t$}\}.$$ If \(\mathbb{P}( I_{t,m})=1\), then by continuity of the probability, \(\mathbb{P}(\mathcal{A}_{t,n-1})\) goes to 1 as \(n\) goes to infinity. 

On the event \(\mathcal{A}_{t,n-1}\), we can find \(\hat{v}_1,\ldots,\hat{v}_m\) in \(\partial \Lambda_{n-1}\) such that, at time \(t\),
\begin{itemize}
\item[($a$')] $\hat{v}_i$ is connected to infinity outside $\Lambda_{n-1}$ for every $i\in[m]$;
\item[($b$')] $\hat{v}_i$ and $\hat{v}_j$ are not connected for $i\neq j$.
\end{itemize}

Condition ($a$') ensures that, for each $i\in[m]$, there exists $v_i\in \partial\Lambda_n$ such that $\langle \hat{v}_i, v_i \rangle$ is open, $v_i$ is connected to infinity outside $\Lambda_{n-1}$, and $ v_i\in \{x \in \partial \Lambda_n \colon x \sim  y, \text{ for some } y \in \Lambda_{n-1}  \} $. Moreover, ($b$') implies ($b$), yielding 
\begin{equation*}
\mathcal{A}_{t,n-1}  \subset \bigcup_{V\subset \Lambda_n}\mathcal{X}_{t,n}^{V}.
\end{equation*}
Hence, for \(n\) sufficiently large, there exists \( V = \{ v_1, \ldots, v_m \} \subset \{x \in \partial \Lambda_n \colon x \sim y \text{ for some } y \in \Lambda_{n-1} \} \) such that
\begin{equation}\label{eq:existeV}
\mathbb{P}(\mathcal{X}_{t,n}^{V}) > 0.
\end{equation}

Next, for a set \(W = \{w_1, \ldots, w_m\} \subset \partial \Lambda_n\), let \( \mathcal{X}_{t,n}^{V,W} \) denote the event that, at time $ t $,  \(w_i\) is connected to \(v_i\) outside \(\Lambda_{n-1}\), and \(w_i\) is connected to infinity outside \(\Lambda_n\) for every $ i \in [m] $. Since there are finitely many possible sets $W\subset\partial\Lambda_n$ with $m$ elements,  \eqref{eq:existeV} implies that $ \mathcal{X}_{t,n}^{V,W} $ has positive probability for some $W$.

Also, let $\Gamma = \{\gamma_1,\ldots,\gamma_m \}$ be a set of $m$ self-avoiding paths in $\Lambda_{n-1}^c$ such that $\gamma_i$ connects $v_i$ to $w_i$ for every $i\in[m]$, and consider the event
\begin{equation*}
    \mathcal{Z}_{t}^{\Gamma} = \{U\in[0,1]^{\E^d}: \gamma_i \mbox{ are open in the configuration } \omega_t(U) \mbox{ for all } i\in[m]\}.
\end{equation*}
Note that the occurrence of $\mathcal{X}_{t,n}^{V,W}\cap \mathcal{Z}_t^{\Gamma}$ implies itens ($a$)-($d$) stated just before Claim \ref{claim:DGamma}. Let us show that $\mathbb{P} \left( \mathcal{X}_{t,n}^{V,W}\cap \mathcal{Z}_t^{\Gamma} \right) > 0$.

For each $j>n-1$, consider the event
\begin{equation*}
\mathcal{Y}_j=\{U\in[0,1]^{\mathbb{E}^d}\colon v_i\longleftrightarrow w_i \mbox{ in } \Lambda_{n-1}^c\cap\Lambda_j \mbox{ for every } i\in[m] \mbox{ on } \omega_t(U)\},
\end{equation*}
and observe that
\begin{equation*}
\lim_{j\longrightarrow\infty}\mathbb{P}\left(\mathcal{X}_{t,n}^{V,W}\cap\mathcal{Y}j\right)=\mathbb{P}\left(\mathcal{X}_{t,n}^{V,W}\right)>0.
\end{equation*}
Fix $j$ large enough such that
\begin{equation}\label{eq:limite}
\mathbb{P}(\mathcal{X}_{t,n}^{V,W}\cap\mathcal{Y}j)>0.
\end{equation}

Let $\mathcal{G}_j$ denote the collection of sets $\Gamma=\{\gamma_1,\ldots,\gamma_m\}$ of $m$ self-avoiding paths in $\Lambda_{n-1}^c\cap\Lambda_j$ such that $\gamma_i$ connects $v_i$ to $w_i$ for all $i\in[m]$. Note that $\mathcal{G}_j$ is finite and
\begin{equation}\label{jgrande}
\mathcal{X}_{t,n}^{V,W}\cap\mathcal{Y}j=\bigcup_{\Gamma\subset\mathcal{G}_j}\left( \mathcal{X}_{t,n}^{V,W}\cap \mathcal{Z}_t^{\Gamma}\right).
\end{equation}
Combining \eqref{eq:limite} with \eqref{jgrande}, we obtain
\begin{equation}\label{eq:quasela}
\mathbb{P} \left( \mathcal{X}_{t,n}^{V,W}\cap \mathcal{Z}_t^{\Gamma} \right) > 0.
\end{equation}

Finally, to account for items ($e$) and ($f$) preceding Claim \ref{claim:DGamma}, let $A\subset\partial^e\Lambda_n$ and define
\begin{equation*}
\mathcal{H}_{t,n}^A =\left\{U\in[0,1]^{\E^d}\colon \omega_t(U,e)=1 \mbox{ if } e\in A \mbox{ and } \omega_t(U,e)=0 \mbox{ if } e\in\partial^e\Lambda_n\cap A^c\right\}.
\end{equation*}
Decompose \(\mathcal{X}_{t,n}^{V,W} \cap \mathcal{Z}_t^{\Gamma}\) as a finite union (over sets $A\subset \partial^e\Lambda_n$) of the events $\mathcal{H}_{t,n}^A$. The claim follows since all the requirements of Claim \ref{claim:DGamma} are satisfied in \( \mathcal{X}_{t,n}^{V,W} \cap \mathcal{Z}_t^{\Gamma} \cap \mathcal{H}_{t,n}^A \) and  by \eqref{eq:quasela}.
\end{proof}

\begin{proof}[Proof of Claim \ref{lem:modif1}] Let us start with the proof of part ($a$). Let $e=\langle x,y \rangle\in\partial^e\Lambda_n$, with $x\in\partial\Lambda_n$. By the definition of $\mathcal{E}_{t,n}^{\Gamma,A}$,  if $\omega_t(U,e) =0$, then $R(e)>t$; whereas if $\omega_t(U,e) =1$, then $R(e)= U(e)$. Furthermore, note that any edge $f$ that is closed in $\omega_t(U)$ will also be closed in $\omega_t(R)$, because in this case $R(f)>t$. Therefore, if $s<R_e$, it follows that
\begin{equation*} 
    \deg_s(x,R)\leq \deg_s(x,U) < k.
\end{equation*}
Hence, $R$ simplifies the boundary of $\Lambda_n$ in $U$ at time  $t$. In particular, \eqref{eq: fora ta igual} holds by Lemma \ref{teo:PMF}.

We now prove ($b$). By the definition of $\mathcal{E}_{t,n}^{\Gamma,A}$, the only edges in $\mathcal{E}(\Lambda_n)$ that are open in $\omega_t(R)$ are those belonging to the paths $\gamma_1$ and $\gamma_2$. Conversely, if the paths $\gamma_1$ and $\gamma_2$ include edges outside $\mathcal{E}(\Lambda_n)$, those edges remain open in $\omega_t(R)$ due to \eqref{eq: fora ta igual}.
\end{proof}

\begin{proof}[Proof of Claim \ref{lem: final}] We may assume $k\geq 3$, since otherwise Proposition 1 in \cite{LSSST} implies $\mathcal{N}_t=0$ almost surely. We first prove ($a$). Let $\gamma$ be as defined in \eqref{eq: defgamma}. We will show that all edges of $\gamma$ are open in $\omega_t(R)$.  

First, take an edge $e = \langle u_i,v_i \rangle$ from $\gamma$, $i \in \{1,2\}$. Observe that the number of open edges incident to $u_i$ is exactly $2$, specifically the two edges that belong to $\gamma$. On the other hand, the number of open edges incident to $v_i$ is at most 3: the edge $e$ itself, the single edge of $\gamma_i$ incident to $v_i$, and possibly the single edge connecting $v_i$ to $\Lambda_{n+1}$ (recall that $v_i$ lies on the surface of $\Lambda_n$).

Now, suppose $e$ is an edge of the self-avoiding path $\alpha$. In this case, each vertex along $\alpha$ is incident to at most two open edges. Consequently, in $\omega_t(R)$, the vertices $v_1$ and $v_2$ are connected.

For the proof of ($b$), it is clear that \eqref{eq: coincidem...} holds for $e \in \mathcal{E}(\partial\Lambda_n)$. Therefore, it suffices to verify it for $e \in \mathcal{E}(\Lambda_n)^c$. We will check that $R$ simplifies the boundary of $U$.  

Suppose $e=\langle x, y \rangle \in \partial^e \Lambda_n$ with $x \in \Lambda_n$. Clearly, $R(e)>t$ whenever $\omega_t(U,e)=0$. Conversely, if $\omega_t(U,e)=1$, then $R(e)=U(e)$. Also, if $x \not\in \{v_1,v_2\}$, then the number of open edges incident on $x$ in  $\omega_s(R,e)$ is the same as in $\omega_s(U,e)$ for all $s<t$. Hence,  for all $s < R(e)$, it holds that
\begin{equation*}
    \deg_s(R,e) = \deg_s(U,e)<k.
\end{equation*}
Finally, if $x \in \{v_1,v_2\}$ and $s<R(e)$, then the number of open edges incident to $x$ at time $s$ is at most 2: one connecting the inside and one belonging to $\gamma_i$, $i=1,2$. 
\end{proof}

\section*{Acknowledgements} Weberson Arcanjo was partially supported by Conselho Nacional de Desenvolvimento Científico e Tecnológico (CNPq), grant 402952/2023-5. Alan Pereira was partially supported by Fundação de Amparo à Pesquisa do Estado de Alagoas (FAPEAL), process E:60030.0000000161/2022, Project APQ2022021000107, and CNPq, grant 402952/2023-5. Diogo dos Santos was partially supported by CNPq, grants 409198/2021-8 and 402952/2023-5. Roger Silva was partially supported by Fundação de Amparo à Pesquisa do Estado de Minas Gerais (FAPEMIG), grant APQ-06547-24.  Marco Ticse was partially supported by FAPEMIG.

\section*{Conflict of interest statement}

The authors have no conflicts of interest to declare.
All co-authors have seen and agree with the contents of the manuscript and
there is no financial interest to report. We certify that the submission is
original work and is not under review at any other publication.

\section*{Data availability statement}

Data sharing not applicable to this article as no datasets were generated or analysed during the current study.

\end{document}